\documentclass[12pt,a4paper]{article}

\usepackage{graphicx}
\usepackage{cite}
\usepackage{authblk}

\usepackage{color}
\usepackage{colortbl}

\usepackage{amsmath,amsfonts,amssymb,mathrsfs,amsthm}
\usepackage{delarray}
\usepackage{enumerate}

\newcommand{\tiff}{if and only if }
\newcommand{\las}{locally asymptotically stable }

\newcommand{\mybeta}{\beta \left(a_1,a_2\right)}
\newcommand{\mygamma}{\gamma \left(a_1,a_2\right)}

\newtheorem{theorem}{Theorem}
\newtheorem{corollary}[theorem]{Corollary}

\newtheorem{remark}[theorem]{Remark}
\newtheorem{proposition}[theorem]{Proposition}

\title{Oscillations in a white blood cell production model with multiple differentiation stages.}

\author[1]{Franziska Knauer}
\author[2]{Thomas Stiehl}
\author[3]{Anna Marciniak-Czochra}

\affil[1]{Institute for Applied Mathematics, Heidelberg University, Im Neuenheimer Feld 205, 69120 Heidelberg, Germany}

\affil[2]{Institute for Applied Mathematics, Interdisciplinary Center for Scientific Computing, Heidelberg University, Im Neuenheimer Feld 205, 69120 Heidelberg, Germany, thomas.stiehl@iwr.uni-heidelberg.de}  
\affil[3]{Institute for Applied Mathematics, Interdisciplinary Center for Scientific Computing, Heidelberg University, Im Neuenheimer Feld 205, 69120 Heidelberg, Germany, anna.marciniak@iwr.uni-heidelberg.de}

\date{}
\begin{document}

\maketitle

\begin{abstract}
In this work we prove occurrence of a super-critical Hopf bifurcation in a model of white blood cell formation structured by three maturation stages. We provide an explicit analytical expression for the bifurcation point depending on model parameters. The Hopf bifurcation is a unique feature of the multi-compartment structure as it does not exist in the corresponding two-compartment model. It appears for a parameter set different from the parameters identified for healthy hematopoiesis and requires changes in at least two cell properties. Model analysis allows identifying a range of biologically plausible parameter sets that can explain persistent oscillations of white blood cell counts observed in some hematopoietic diseases. Relating the identified parameter sets to recent experimental and clinical findings provides insights into the pathological mechanisms leading to oscillating blood cell counts.\\

\end{abstract}

\section{Introduction}

This work is devoted to the study of Hopf bifurcations and emergence of oscillatory dynamics in a multi-compartmental model of healthy blood cell production (hematopoiesis). The model describes a multi-stage blood cell production process based on self-renewal and differentiation of stem and progenitor cells, which is needed for regeneration of mature white blood cells. Each maturation stage is treated as a homogeneous compartment and its time evolution is described by an ordinary differential equation with coefficients controlled by a nonlinear feedback signal that depends on the count of mature cells. The model was introduced in ref.  \cite{SCDev} and then has been applied to study blood cell recovery after bone marrow transplantation \cite{BMT,AdvExpMedBiol} and extended to model evolution and response to therapy of hematological diseases such as acute leukemias  \cite{SciRep,MMNP,Interface,CR} and myelodysplastic syndromes \cite{Walenda}. Although mathematical understanding of the underlying equations proved to be useful for model applications and interpretation in context of the patients' data \cite{Busse,Interface}, rigorous analysis of the underlying equations has been established only in the case of a two-compartment maturation structure in ref. \cite{Nakata}. In this work, we close the gap and provide analysis of a system involving an intermediate differentiation stage given by a three-compartment structure. While in case of the two-compartment model, the positive equilibrium is globally stable whenever it exists \cite{Nakata}, our analysis shows that increasing the number of compartments may lead to the loss of stability of the positive equilibrium due to a super-critical Hopf bifurcation. This finding is of biological relevance, since it shows that the number of maturation stages may impact the system dynamics. \\ 
 
 Periodic oscillations in a model with non-linear feedback mechanisms but without explicit delays have not been studied in the context of hematopoietic system so far. Our model shows that we can have cycling hematopoiesis as inherent property of the multistep maturation process, however, arising far away from the parameter regime corresponding to the healthy system. Periodic oscillations of blood cell counts are a rare but intriguing phenomenon that can be observed in humans and animals \cite{Guerry,Dale,Dale2}. Cyclic neutropenia is the most frequent disease with oscillating blood cell counts.  In cyclic neutropenia patients' neutrophil counts show periodic oscillations with maxima that are significantly below the neutrophil counts of healthy individuals. Since neutrophils are responsible for immune defence, patients repeatedly suffer from infections \cite{Dale3}. The disease can be cured by transplantation of healthy bone marrow \cite{Okolo,Dale4}. Similarly, accidental transplantation of bone marrow from a patient with cyclic neutropenia transfers the disorder to a previously unaffected host as has been shown in ref.  \cite{Krance}.  A mechanistic understanding of the disease, therefore, requires quantitative insights into blood cell formation and its regulations.\\

Cyclic neutropenia has been extensively studied using mathematical models. One hypothesis derived from mathematical models is that oscillations are caused by increased apoptosis /reduced proliferation of neutrophil precursors \cite{Bernard,Lei} combined with a reduced entry of stem cells into the proliferative phase \cite{Colijn}. An alternative mechanism could be an increase of the death rates of stem cells \cite{Lei,Mackey}. Other model-derived hypotheses for the origin of cyclic neutropenia include reduced maturation speed \cite{Wheldon} or dysfunction of feedback mechanisms \cite{Schulthess}. Experimental studies  suggest abnormal responsiveness of cells to growth factors \cite{Hammond,Wright}
or increased apoptosis of progenitor cells \cite{Grenda} as possible reasons for the origin of periodic oscillations. One common feature of most models of cyclic neutropenia is that they include  constant or distributed delays. Intuitively, a system with feedbacks the effects of which occur with a delay, can be supposed to oscillate if the feedback loop gain is large enough. However, whether oscillations indeed appear, depends on configuration of feedbacks. In ref. \cite{Dingli} it has been shown that a reduction of progenitor cell's self-renewal is sufficient to explain oscillatory dynamics in a model with linear feedback regulation and without delays. However, a system of linear compartments effectively constitutes a delay distributed according to a convolution of negative exponential functions, i.e. non-central gamma type, \cite{Johnson}(Section 17.8.7). Other mathematical modeling works studying cyclic neutropenia include  \cite{King-Smith,Kazarinoff,Haurie,Gopalsamy}. Different modeling approaches are reviewed in \cite{Haurie2,Colijn2}. \\

The paper is organized as follows. In Section 2 we present a derivation of the considered model and its biological justification. In Section 3 we provide analytical results, including uniform boundedness of solutions and linear stability analysis. We provide criteria for the occurrence of a Hopf bifurcation and illustrate them by model simulations. In Section 4 we study systematically for which subsets of the biologically relevant parameter space Hopf bifurcations occur and we relate our findings to experimental results. Section 5 concludes with a short summary and a discussion of the obtained results.   

\section{Model motivation and formulation}\label{sec:model}

Blood cells are continuously produced during the life of higher metazoans. This task is fulfilled by the hematopoietic (blood forming) system which is located in the bone marrow \cite{Reya}. Hematopoietic stem cells (HSC) give rise to progenitor cells which subsequently produce mature cells \cite{Reya}. Due to its vital importance hematopoiesis is a tightly regulated process. Complex nonlinear feedback mechanisms allow the organisms to adapt to environmental conditions and to efficiently respond to perturbations such as blood loss or infection \cite{Metcalf}. Key processes during hematopoiesis are cell proliferation, self-renewal and differentiation. Proliferation denotes the division of one parent cell into two progeny. If progeny are of the same cell type as the parent, e.g., a progeny of a stem cell is again a stem cell, this process is referred to as self-renewal. The alternative scenario, where progeny are of a more mature cell type compared to their parent cell is referred to as differentiation \cite{COISB,SCREP}.\\

 In this work, we study for which configurations of proliferation and self-renewal parameters periodic oscillations of blood cell counts can occur. We focus on a three-compartment version of the model describing dynamics of stem cells, progenitor cells and mature cells. Dynamics of each cell population is  described by an ordinary differential equation. Denoting the number of cells per kg of body weight at time $t$ as $u_i(t)$, where $i=1$ corresponds to stem cells, $i=2$ to progenitor cells and $i=3$ to mature white cells,  each cell type is characterized by the following parameters: 
\begin{itemize}
	\item Proliferation rate $p_i$, describing the frequency of cell divisions per unit of time. In accordance with biology we assume that mature cells do not divide \cite{Jandl}.
	\item Fraction of self-renewal $a_i$,
	describing the fraction of progeny cells originating from division and  returning to the compartment of their parent cell.
	\item Death rate $d_i$, describing the fraction of cells dying per
	unit of time. For simplicity, we assume that immature
	cells ($i = 1, 2$) do not die and that mature cells die at
	constant rates. This is a good approximation  of reality \cite{BMT,Jandl}, 
\end{itemize}

 {White blood cell production is regulated by negative feedback signals, such as G-CSF \cite{Metcalf,Layton}. Since signal dynamics take place on a faster time scale compared to cell divisions, a quasi-steady state approximation can be used to describe the signal concentration as a function of white blood cell counts \cite{SCDev,MCM,M2AS}.
 $$s(t)=\frac{1}{1+k{u_3(t)}}, $$
{where $k>0$, see \cite{SCDev} for details.} Rigorous proof of the corresponding quasi-steady state model reduction is presented in ref.  \cite{M2AS}.\\

Following the previous work \cite{SCDev,BMT,AdvExpMedBiol}, we assume feedback inhibition of the fraction of self-renewal by mature cells, i.e. $a_i(t) = a_is(t)$.\\

The flux to division of healthy cells in compartment $i$  at time $t$ equals $p_i{u}_i(t)$.
During division, a parent cell is replaced by two progeny cells. The outflux from mitosis at time $t$, therefore, equals $2p_i {u}_i(t)$,
of which the fraction $2a_i(t)p_i {u}_i(t)$ stays in compartment $i$ (referred to as self-renewal). The fraction $2(1-a_i(t))p_i {u}_i(t)$ proceeds to maturation stage $i+1$ (process referred to as differentiation). Taking into account that mature cells do not divide and that the parent cell disappears as it gives rise to its progeny, we obtain the following system of differential equations

\begin{align*}
\frac{du_1}{dt} & = \left(2\,\frac{a_1}{1+k\,u_3}-1\right)\,p_1\,u_1 \tag{M1} \label{M1}\\
\frac{du_2}{dt} & = \left(2\,\frac{a_2}{1+k\,u_3}-1\right)\,p_2\,u_2 + 2\,\left(1-\frac{a_1}{1+k\,u_3}\right)\,p_1\,u_1 \tag{M2} \label{M2}\\
\frac{du_3}{dt} & = 2\,\left(1-\frac{a_2}{1+k\,u_3}\right)\,p_2\,u_2 - d_3\,u_3, \tag{M3} \label{M3}
\end{align*}

where $p_1, p_2 >0, d_3 >0$ and $k > 0$. The initial conditions fulfill $u_1(0)>0$, $u_2(0)\geq 0$, $u_3(0)\geq 0$. A schematic of the model is depicted in Figure \ref{figModel}.\\

\begin{figure}%
	\centering
	\includegraphics[width=\columnwidth]{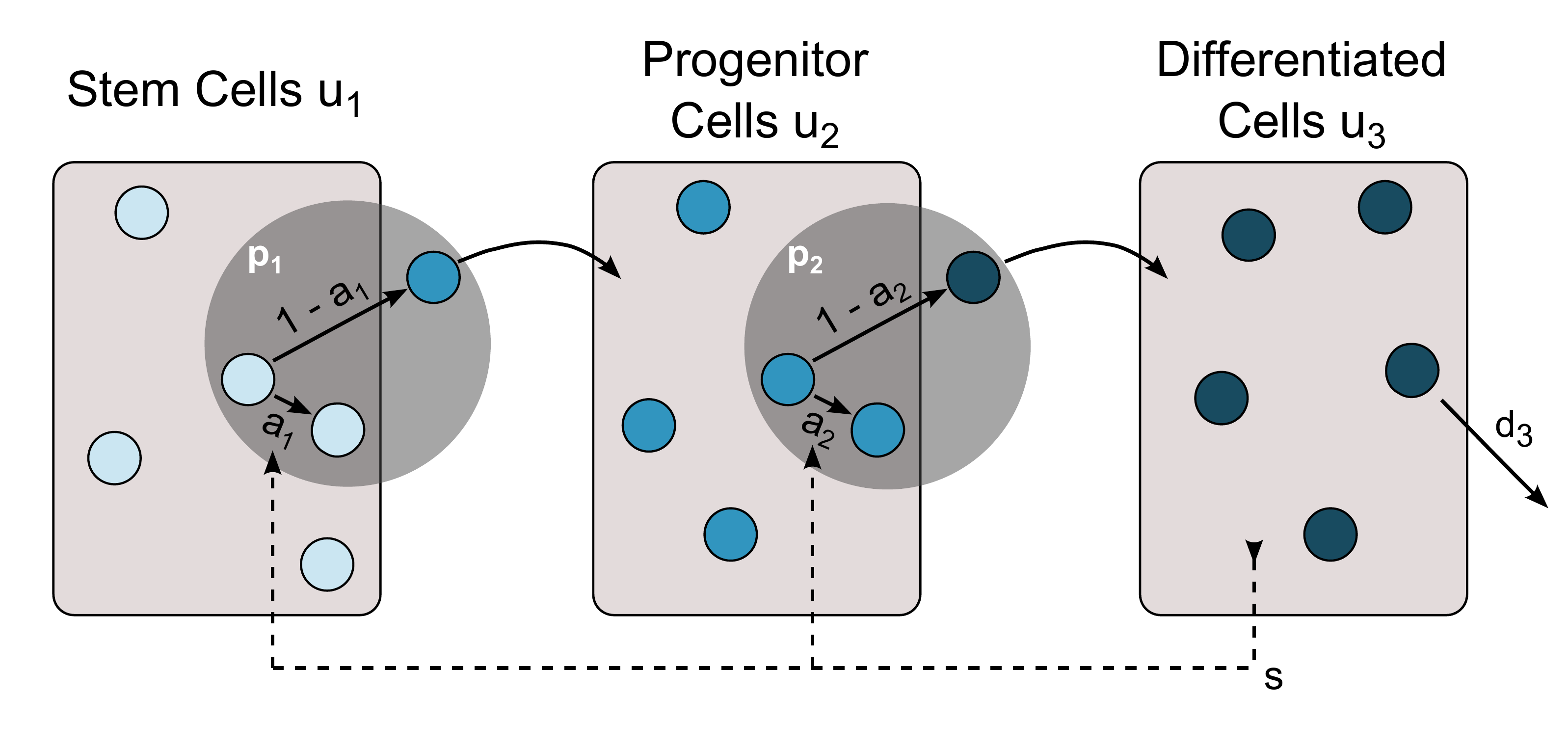}%
	\caption[Scheme of the cell fluxes.]{Scheme of the model. $p_{1,2}$ denote the proliferation rates, $a_{1,2}$ the fractions of self-renewal, $d_{3}$ the death rate, and s the feedback-signal.}%
	\label{figModel}%
\end{figure}

In order to simplify our calculations, we finally rewrite the model equations in dimensionless terms using the reparametrizations $\tilde{t} := p_1\,t$, $\tilde{p}_2 := \frac{p_2}{p_1}$, $\tilde{d_3} := \frac{d_3}{p_1}$, and $\tilde{u}_i\left(\tilde{t}\right) := u_i\left(\frac{\tilde{t}}{p_1}\right)$ for $i = 1, 2, 3$:
\begin{align*}
\frac{d\tilde{u}_1}{dt} & = \left(2\,\frac{a_1}{1+k\,\tilde{u}_3}-1\right)\,\tilde{u}_1 \tag{M1*} \label{M1*}\\
\frac{d\tilde{u}_2}{dt} & = \left(2\,\frac{a_2}{1+k\,\tilde{u}_3}-1\right)\,\tilde{p}_2\,\tilde{u}_2 + 2\,\left(1-\frac{a_1}{1+k\,\tilde{u}_3}\right)\,\tilde{u}_1 \tag{M2*} \label{M2*}\\
\frac{d\tilde{u}_3}{dt} & = 2\,\left(1-\frac{a_2}{1+k\,\tilde{u}_3}\right)\,\tilde{p}_2\,\tilde{u}_2 - \tilde{d_3}\,\tilde{u}_3 \tag{M3*} \label{M3*}
\end{align*}
The parameter ranges and initial conditions remain unchanged. For convenience, we drop the \textasciitilde-symbol in the remainder of this paper.

\section{Model analysis}\label{sec:results}

In this section we show uniform boundedness of solutions (Section \ref{sub:existence}), provide conditions for existence of non-negative steady states (Section \ref{sub:steadystates}) and linearized stability analysis (Section \ref{sub:stability}), and prove occurrence of Hopf bifurcation (Section \ref{sub:hopf}).
\subsection{Uniform boundedness of solutions}\label{sub:existence}

\begin{theorem}\label{thm:ExAndBoundedness}
	Solutions of \eqref{M1}-\eqref{M3} with positive initial values remain in the first octant and for sufficiently large times $t$ even in a compact subset $C$ which does not depend on the initial values. 
\end{theorem}

\begin{proof}
The proof follows the lines of the proof in ref. \cite{Busse}, adjusted to the three-compartment structure of the model. We consider the rescaled system \eqref{M1*}-\eqref{M3*} and define $C$ as the cuboid bounded by the planes $\lbrace (u_1,u_2,u_3) \in \mathbb{R}^3 \mid u_i=0 \rbrace$ and $\lbrace (u_1,u_2,u_3) \in \mathbb{R}^3 \mid u_i=C_i \rbrace$ for $i = 1,2,3$. Obviously, the orbits of solutions with nonnegative initial values remain in the first octant. To show uniform boundedness from above, we compute equations for the fractions $v_1 := \frac{u_1}{u_2}$ and $v_2 := \frac{u_2}{u_3}$ that lead to the following estimates,
	\begin{align*}
	\frac{dv_1}{dt} &< \left[1 + p_2 - 2 \, \left(1-a_1\right)\,v_1\right]\,v_1 < 0 \quad &&\text{ for } v_1 \geq \frac{1+p_2}{2 \, \left(1-a_1\right)} =: B_1\\
	\frac{dv_2}{dt} &< \left[p_2 + d_3 + 2\,B_1-2\left(1-a_2\right)\,v_2\right]\,v_2 < 0 \quad &&\text{ for } v_2 \geq \frac{p_2+d_3+2\, B_1}{2 \, \left(1-a_2\right)} =: B_2,
	\end{align*}
for all $u_3 \geq 0$.
	Taking $M_i:=\max\{B_i,v_i(0)\}$, we conclude $u_2 \geq \frac{1}{M_1} \,u_1$ and $u_3 \geq \frac{1}{M_2}u_2 \geq u_1 \, \frac{1}{M_2\,M_1}$. With these relations we obtain 
	\begin{equation*}
	\begin{gathered}
	\frac{du_1}{dt} \leq \left(2\,\frac{a_1}{1+\frac{k}{M_2\,M_1}\,u_1}-1\right)\,u_1<0,\;\; \text{ for } \;\; u_1 > 	\frac{M_2\,M_1}{k}\,\left(2\,a_1-1\right)=:K_1
	\end{gathered}
	\end{equation*}
	and 
	\begin{equation*}\label{u2}
	\begin{gathered}
	\frac{du_2}{dt} \leq \left(2\,\frac{a_2}{1+\frac{k}{M_2}\,u_2}-1\right)\,p_2\,u_2 + 2\, K_1<0\\
	\text{for } u_2 > \operatorname{max}\left\{\frac{4\,a_2-1}{k}\,M_2, \frac{4}{p_2}\,K_1  \right\}=:K_2
	\end{gathered}
	\end{equation*}
	as well as
	\begin{equation*}
	\begin{gathered}
	\frac{du_3}{dt} \leq 2\,p_2\,K_2 -d_3\,u_3< 0 \\
	\text{for } u_3>\frac{2\,p_2\,K_2}{d_3}=:K_3.
	\end{gathered}
	\end{equation*}
Taking $C_i=\max\{K_i, u_i(0)\}$, we conclude about invariance of the set $C$.
	
\end{proof}

\subsection{Existence of steady states}\label{sub:steadystates}

Existence and uniqueness of steady states has been systematically studied in ref. \cite{MCM}. The following proposition summarizes the results. 

\begin{proposition}\label{Prop1}
	\begin{enumerate}
		\item The trivial steady state $E_0 = \left(\bar{u}_1^0,\bar{u}_2^0,\bar{u}_3^0\right)^{\mathrm T} = \left(0,0,0\right)^{\mathrm T}$ of (\ref{M1}) - (\ref{M3}) exists for all parameter values.
		\item There exists a semi-trivial steady state $E_1 = \left(0,\bar{u}_2^1,\bar{u}_3^1\right)^{\mathrm T}$ of (\ref{M1}) - (\ref{M3}) with positive components $\bar{u}_2^1,\bar{u}_3^1$ given by
		\begin{align*}
		E_1 = \left(0,\frac{d_3}{p_2}\,\bar{u}_3^1,\frac{2\,a_2-1}{k}\right)^{\mathrm T}
		\end{align*}
		if and only if
		\[
		a_2 > \frac{1}{2}.
		\]
		\item There exists a strictly positive steady state $E_2 = \left(\bar{u}_1^2,\bar{u}_2^2,\bar{u}_3^2\right)^{\mathrm T}$ of (\ref{M1}) - (\ref{M3}) given by
		\begin{align*}
		E_2 = \left(\left(1-\frac{a_2}{a_1}\right)\,\frac{p_2}{p_1}\,\bar{u}_2^2, 
		\frac{d_3}{\left(2-\frac{a_2}{a_1}\right)\,p_2}\,\bar{u}_3^2, \frac{2\,a_1-1}{k}\right)^{\mathrm T}
		\end{align*}
		if and only if
		\begin{align*}
		a_1 > \frac{1}{2} \text{ and } a_2 < a_1. 
		\end{align*}
	\end{enumerate}
	
\end{proposition}

Let us remark that for appropriate values of $d_3$, $p_2, \text{ and } k$ the steady state $E_2$ can be any point in $\mathbb{R}^3_+$. We have $E_2=(\bar{u}_1, \bar{u}_2, \bar{u}_3)$ if 
\begin{align*}
k = \frac{2\,a_1-1}{\bar{u}_3}, \quad d_3 = \frac{2-\frac{a_2}{a_1}}{1 - \frac{a_2}{a_1}} \, \frac{\bar{u}_1}{\bar{u}_3}p_1, \quad	p_2 = \frac{d_3}{2 - \frac{a_2}{a_1}} \, \frac{\bar{u}_3}{\bar{u}_2}.
\end{align*}


\subsection{Linear asymptotic stability}\label{sub:stability}
In  Proposition \ref{prop:stability} we summarize the linear asymptotic stability of the steady states $E_0$ and $E_1$. In Theorem \ref{thm:stabE2} we study  the linear asymptotic stability of $E_2$ using the Routh-Hurwitz Criterion.\\

\begin{proposition}\label{prop:stability}
	\begin{enumerate}[(i)]
		\item 	The steady state $E_0=\left(0,0,0\right)^{\mathrm T}$ is \las if $\operatorname{max}\lbrace a_1, a_2 \rbrace < \frac{1}{2}$ and unstable if $\operatorname{max}\lbrace a_1, a_2 \rbrace > \frac{1}{2}$.
		\item 	The steady state $E_1=\left(0,\frac{d_3}{p_2}\,\frac{2\,a_2-1}{k},\frac{2\,a_2-1}{k}\right)^{\mathrm T}$ is \las if $a_1<a_2$ and unstable if $a_1>a_2$.
	\end{enumerate}
\end{proposition}

\begin{proof}  
	The Jacobian of system  \eqref{M1*}-\eqref{M3*} at the steady state $\left(\bar{u}_1^i, \bar{u}_2^i, \bar{u}_3^i\right)$is given by
	
	\[\begin{split} &J\left(\bar{u}_1^i, \bar{u}_2^i, \bar{u}_3^i\right) = \\
	&\quad \left(
	\begin{matrix}
	2\,\frac{a_1}{1+k\,\bar{u}_3^i}-1&
	0&
	-2\,a_1\,k\,\frac{1}{\left(1+k\,\bar{u}_3^i\right)^2}\,\bar{u}_1^i \cr
	2\,\left(1-\frac{a_1}{1+k\,\bar{u}_3^i}\right)&
	\left(2\,\frac{a_2}{1+k\,\bar{u}_3^i}-1\right)\,p_2&
	-2\,p_2\,a_2\,k\,\frac{1}{\left(1+k\,\bar{u}_3^i\right)^2}\,\bar{u}_2^i+ 2\,a_1\,k\,\frac{1}{\left(1+k\,\bar{u}_3^i\right)^2}\,\bar{u}_1^i \cr
	0&
	2\,\left(1-\frac{a_2}{1+k\,\bar{u}_3^i}\right)\,p_2&
	2\,p_2\,a_2\,k\,\frac{1}{\left(1+k\,\bar{u}_3^i\right)^2}\,\bar{u}_2^i-d_3
	\end{matrix}
	\right).
	\end{split}\]
	$ $\\
	(i) Consider the Jacobian matrix at $E_0$
	\[J\left(0,0,0\right) =
	\left(
	\begin{matrix}
	2\,a_1-1&
	0&
	0\cr
	2\,\left(1-a_1\right)&
	\left(2\,a_2-1\right)\,p_2&
	0 \cr
	0&
	2\,\left(1-a_2\right)\,p_2&
	-d_3
	\end{matrix}
	\right).
	\]
	As $J\left(0,0,0\right)$ is a lower triangular matrix, we obtain the eigenvalues
	\begin{align*}
	\lambda_1^0 &= 2\,a_1-1 \\
	\lambda_2^0 &= \left(2\,a_2-1\right)\,p_2 \\
	\lambda_3^0 &= -d_3,
	\end{align*}
	which implies (i).\\
	
	(ii) Consider the Jacobian matrix at $E_1$
	\[J\left(\bar{u}_1^1,\bar{u}_2^1,\bar{u}_3^1\right)=
	\left(
	\begin{matrix}
	\frac{a_1}{a_2}-1&
	0&
	0 \cr
	2-\frac{a_1}{a_2}&
	0&
	-\left(1-\frac{1}{2\,a_2}\right)\,d_3\cr
	0&
	p_2&
	-\frac{1}{2\,a_2}\,d_3
	\end{matrix}
	\right).
	\]
	We recall that for existence of $E_1$ it has to hold $a_2>1/2$. We obtain the characteristic equation
	\[\left(\lambda - \frac{a_1}{a_2} +1\right) \, \left[\lambda \, \left(\lambda + \frac{1}{2\,a_2}\,d_3 \right)+ \left(1-\frac{1}{2\,a_2}\right)\,d_3 \, p_2\right] =0
	\]
	and thus the eigenvalues
	\begin{align*}
	\lambda_1^1 &= \frac{a_1}{a_2} -1 \\
	\lambda_2^1 &= -\frac{1}{4\,a_2}\,d_3 + \sqrt{\left(\frac{1}{4\,a_2}\,d_3 \right)^2-\left(1-\frac{1}{2\,a_2}\right)\,d_3 \, p_2}\\
	\lambda_3^1 &= -\frac{1}{4\,a_2}\,d_3 - \underbrace{\sqrt{\left(\frac{1}{4\,a_2}\,d_3 \right)^2-\left(1-\frac{1}{2\,a_2}\right)\,d_3 \, p_2}}_{<\frac{1}{4\,a_2}\,d_3 \text{ or imaginary, since $a_2>0.5$}}
	\end{align*}
	As $\lambda_2^1$ and $\lambda_3^1$ have always negative real parts, the only condition that needs to hold for $E_1$ to be \las is $a_1<a_2$. If $a_1>a_2$, $E_1$ is unstable. 
	
\end{proof}



For $E_2$, we will show local asymptotic stability using the Routh-Hurwitz-Criterion.

{Using the expression for $E_2$ from Proposition \ref{Prop1}} we can prove the following

\begin{theorem}\label{thm:stabE2}
	The positive steady state $E_2$ of system \eqref{M1} - \eqref{M3} is \las if 
	\begin{align}
	p_2> \frac{1}{1-\frac{a_2}{a_1}} \, \left[\frac{1}{\mygamma} - \mybeta \, \frac{d_3}{p_1} \right]p_1. \label{stabcrit}
	\end{align}
	$E_2$ unstable if
	\begin{align}
	p_2< \frac{1}{1-\frac{a_2}{a_1}} \, \left[\frac{1}{\mygamma} - \mybeta \, \frac{d_3}{p_1} \right]p_1, \label{instabcrit}
	\end{align}
	where
	\begin{align*}
	\mybeta &= 1-\frac{a_2}{a_1}\,\left(1-\frac{1}{2\,a_1}\right)\,\frac{1}{2-\frac{a_2}{a_1}}\\
	\mygamma &= \frac{1}{2\,a_1}\frac{1}{1-\frac{1}{2\,a_1}} + 	\frac{a_2}{a_1}\,\frac{1}{\left(2-\frac{a_2}{a_1}\right)\,\left(1-\frac{a_2}{a_1}\right)}.
	\end{align*}
\end{theorem}
\begin{proof}
	We perform calculations for the transformed system \eqref{M1*}-\eqref{M3*}.
	
	Consider the Jacobian matrix at $E_2$, i.e.\
	\begin{align*}
	&J\left(\bar{u}_1^2,\bar{u}_2^2,\bar{u}_3^2\right)\\
	\\ \displaybreak[0]
	=& \left(
	\begin{matrix}
	0&
	0&
	-\left(1-\frac{1}{2\,a_1}\right)\,\left(1-\frac{a_2}{a_1}\right)\,\frac{1}{2-\frac{a_2}{a_1}}\,d_3 \cr
	1&
	-\left(1-\frac{a_2}{a_1}\right)\,p_2&
	\left(1-\frac{1}{2\,a_1}\right)\,\left(1-2\,\frac{a_2}{a_1}\right)\,\frac{1}{2-\frac{a_2}{a_1}}\,d_3 \cr
	0&
	\left(2-\frac{a_2}{a_1}\right)\,p_2&
	\left(\frac{a_2}{a_1}\,\left(1-\frac{1}{2\,a_1}\right)\,\frac{1}{2-\frac{a_2}{a_1}}-1\right)\,d_3
	\end{matrix}
	\right).
	\end{align*}
	We obtain the characteristic equation
	\begin{align*}
	0= & \quad \lambda^3 + \underbrace{\left[\left(1-\frac{a_2}{a_1}\right)\,p_2+\left(1-\frac{a_2}{a_1}\,\left(1-\frac{1}{2\,a_1}\right)\,\frac{1}{2-\frac{a_2}{a_1}}\right)\,d_3\right]}_{=:b_1}\,\lambda^2 \\
	&+\underbrace{\left[\left(1-\frac{a_2}{a_1}\right)\,\left(1-\frac{a_2}{a_1}\,\left(1-\frac{1}{2\,a_1}\right)\,\frac{1}{2-\frac{a_2}{a_1}}\right) - \left(1-\frac{1}{2\,a_1}\right)\,\left(1-2\,\frac{a_2}{a_1}\right)\right]\,d_3\,p_2}_{=:b_2}\,\lambda\\ 		&+\underbrace{\left(1-\frac{1}{2\,a_1}\right)\,\left(1-\frac{a_2}{a_1}\right)\,d_3\,p_2}_{=:b_3}.
	\end{align*}

	We observe that under positivity conditions for $E_2$ ($a_1>a_2$ and $a_1>\frac{1}{2}$) the relations $b_1>0$ and $b_3>0$ hold true: $b_3$ is a product with positive factors only and therefore positive itself. The expression $b_1$ can be written as $b_1 = \left(1-\frac{a_2}{a_1}\right)\, p_2 + \left(1-P\right)\,d_3$ where $P$ is a product consisting of factors that are in $(0,1)$ under positivity conditions. Thus, $1-P$ is positive and, therefore, $b_1$ is, as a sum of two positive summands, positive as well.\\
	
	We distinguish between the following parameter configurations. Details can be found in the book by Gantmacher \cite{Gantm}.
	
	\begin{align*}
	b_1\,b_2-b_3 >0 \, \Leftrightarrow &\text{ $\lambda_1, \lambda_2, \lambda_3$ have negative real parts.}\\
	b_1\,b_2-b_3 =0 \, \Leftrightarrow &\text{ There is one eigenvalue with negative real part and a couple of}\\
	&\text{ complex conjugated eigenvalues with zero real parts.}\\
	b_1\,b_2-b_3 <0 \, \Leftrightarrow &\text{ There is one eigenvalue with negative and two with positive  real parts.}\\
	\end{align*}

	It remains to determine conditions so that the relations $b_1\,b_2 - b_3 > 0$ and $b_1\,b_2 - b_3 < 0$ respectively are satisfied, to complete the proof. Using $\mybeta$ and $\mygamma$ as defined in the theorem, we can rearrange $b_1\,b_2 - b_3$ as follows. Further details on how to proceed can be found in the appendix \ref{app:proofsstab}.
	\begin{align}
	\label{Hurwitz}
	& \, b_1\,b_2 - b_3\\\notag
	&= \left(1-\frac{1}{2\,a_1}\right)\,\left(1-\frac{a_2}{a_1}\right)\,\left(\left[\left(1-\frac{a_2}{a_1}\right)\,p_2 +  		\mybeta \,d_3\right] \cdot \mygamma -1\right)\,d_3\,p_2.
	\end{align}
	
	As the factors $\left(1-\frac{1}{2\,a_1}\right)$ and $\left(1-\frac{a_2}{a_1}\right)$ are positive under positivity conditions for $E_2$ and $d_3$ and $p_2$ are positive anyway, it suffices to find conditions which ensure that the remaining factor is of positive sign. We find
	\begin{align}\
	\notag 0&<\left[\left(1-\frac{a_2}{a_1}\right)\,p_2 + \mybeta\,d_3\right] \cdot \mygamma -1 \\ 
	\Leftrightarrow \quad  p_2&> \frac{1}{1-\frac{a_2}{a_1}} \, \left[\frac{1}{\mygamma} - \mybeta \, d_3 \right] \label{stable}
	\end{align}
	Thus, the eigenvalues of the Jacobian matrix at $E_2$ have negative real parts \tiff relation (\ref{stable}) holds. This implies local asymptotic stability. In analogy there exist eigenvalues with positive real parts \tiff $b_1\,b_2 - b_3<0$. Transforming back to the parameters of system \eqref{M1}-\eqref{M3} completes the proof of this theorem.
\end{proof}
Finally, we will see that the parameter range where $E_2$ exists and is unstable is bounded:

\begin{proposition}
	For $p_1=1$ the set 
	\[A=\left\{ \left(a_1,a_2,d_3,p_2\right) \in (0,1)^2\times\mathbb{R}^2_+\mid a_1>\frac{1}{2}, a_1>a_2,  E_2\text{ is unstable}\right\},\]
	i.e.\ the parameter range where $E_2$ exists, is positive and is unstable, is a bounded subset of the parameter space.
\end{proposition}
\begin{proof}
	First, we note that for all $\left(a_1,a_2,d_3,p_2\right) \in A$ it holds that
	\[p_2 \leq \frac{1}{1-\frac{a_2}{a_1}} \, \left[\frac{1}{\mygamma} - \mybeta \, d_3 \right]\]
	and thus
	\[A \subseteq \bigcup_{a_1 \in (\frac{1}{2},1),a_2 \in (0,a_1)} \lbrace a_1\rbrace \times \lbrace a_2 \rbrace \times A_{a_1,a_2},\]
	where
	\[A_{a_1,a_2} = \lbrace \left(d_3,p_2 \right) \in \mathbb{R}^2_+ \mid p_2 \leq \frac{1}{1-\frac{a_2}{a_1}} \, \left[\frac{1}{\mygamma} - \mybeta \, d_3 \right]\rbrace.\]
	As the range of $a_1$ and $a_2$ is bounded anyway, it suffices to show that the sets $A_{a_1,a_2}$ are uniformly bounded. To this end, we observe that for fixed $a_1$ and $a_2$ the boundary of each $A_{a_1,a_2}$ consists of the part of the graph of the linear equation
	\[p_2 = \frac{1}{1-\frac{a_2}{a_1}} \, \left[\frac{1}{\mygamma} - \mybeta \, d_3 \right]\]
	lying in the first quadrant and its axis intercepts (see Figure \ref{A_zweidim}).
	\begin{figure}[!hbt]
		\includegraphics[width=\columnwidth]{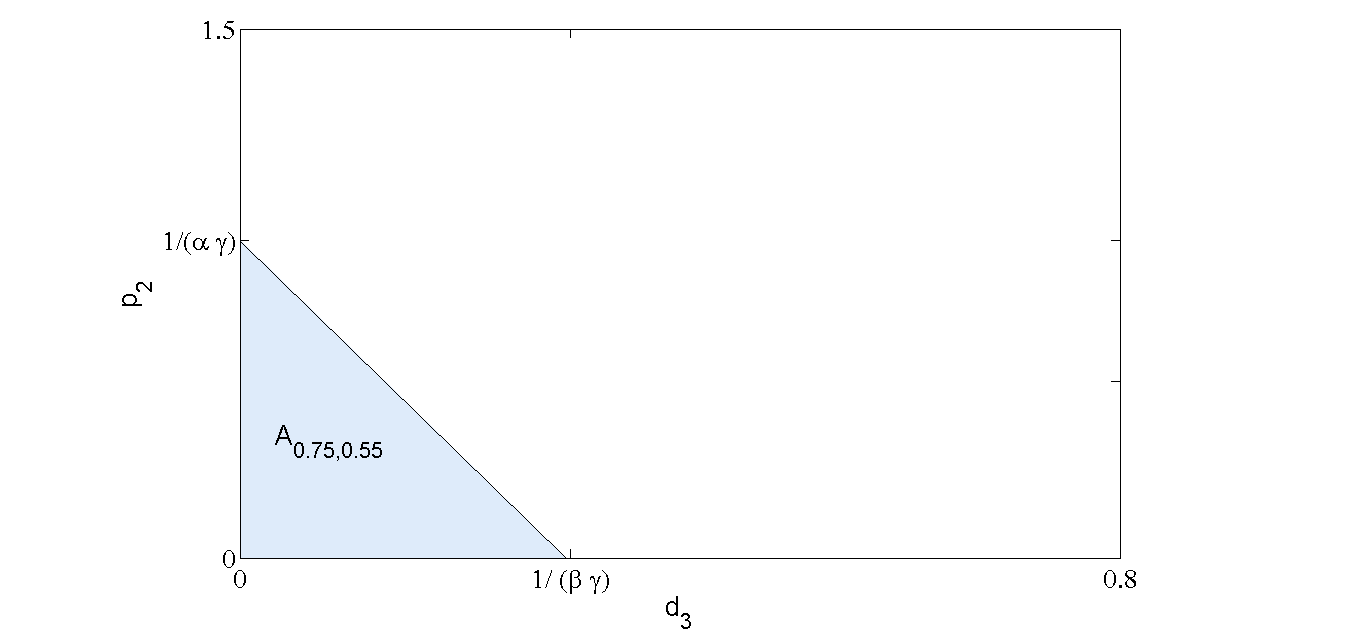}
		\caption[The set $A_{a_1,a_2}$ for $a_1=0.75$ and $a_2=0.55$.]{The set $A_{a_1,a_2}$ for $a_1=0.75$ and $a_2=0.55$. $\alpha = 1 - \frac{a_2}{a_1}$}
		\label{A_zweidim}
	\end{figure}
	We will show that the intercepts of the $d_3$- and the $p_2$-axis $\frac{1}{\mybeta \, \mygamma}$ and $\frac{1}{(1-\frac{a_2}{a_1})\cdot\mygamma}$ respectively are uniformly bounded. For $p_2 \in A_{a_1,a_2}$ for all $a_1, a_2 \in (0,1)$ it holds:
	\begin{multline*}
	\frac{1}{p_2}\geq\mygamma\,\left(1-\frac{a_2}{a_1}\right) >\frac{1}{2}\, \left[\frac{1-\frac{a_2}{a_1}}{1-\frac{1}{2\,a_1}} + \frac{a_2}{a_1}\right]= \frac{1}{2} \, \left[1+\frac{1}{2\,a_1}\,\frac{1-\frac{a_2}{a_1}}{1-\frac{1}{2\,a_1}}\right]
	> \frac{1}{2}.
	\end{multline*}
	This is equivalent to $p_2<2$ for all $a_1,a_2 \in (0,1)$. Furthermore, we note that
	\begin{align*}
	\mygamma \, \mybeta >\mygamma \, \left(1-\frac{a_2}{a_1}\right) > \frac{1}{2}
	\end{align*}
	and thus, we also have for all $d_3 \in A_{a_1,a_2}$ $d_3\leq \frac{1}{\mybeta \, \mygamma}<2 \ \forall a_1,a_2 \in (0,1)$. This means that the set $A_{a_1,a_2}$ are uniformly bounded and therefore, the set A consisting of all parameters for which $E_2$ is unstable is bounded.
\end{proof}

Table \ref{tabelle} summarizes existence and local asymptotic stability of the equilibria $E_0$ to $E_2$.

\begin{table}[!htb]
	\centering
	\begin{tabular}{| m{3cm}| m{2.5cm} | m{2.5cm}|}
		\hline
		\multicolumn{1}{|r||}{} & \multicolumn{1}{c|}{$a_1>a_2$} & \multicolumn{1}{c|}{$a_1<a_2$} \\
		\hline \hline
		\multicolumn{1}{|r||}{$a_1<\frac{1}{2}, a_2<\frac{1}{2}$} & $E_0$: stable\newline
		$E_1: \nexists$ \newline
		$E_2: \nexists$ & $E_0$: stable\newline
		$E_1: \nexists$ \newline
		$E_2: \nexists$\\
		\hline
		\multicolumn{1}{|r||}{$a_1<\frac{1}{2}, a_2>\frac{1}{2}$} & &$E_0$: unstable \newline $E_1$: stable\newline $E_2: \nexists$\\
		\hline
		\multicolumn{1}{|r||}{$a_1>\frac{1}{2}, a_2<\frac{1}{2}$} & $E_0$: unstable \newline $E_1: \nexists$ \newline$E_2$: exists &  \\
		\hline
		\multicolumn{1}{|r||}{$a_1>\frac{1}{2}, a_2>\frac{1}{2}$} & $E_0$: unstable \newline $E_1$: unstable \newline $E_2$: exists & $E_0$: unstable \newline $E_1$: stable \newline $E_2: \nexists$\\
		\hline
	\end{tabular}
	\caption[Summary of existence and stability conditions for the steady states depending on the parameter values of $a_1$ and $a_2$.]{Summary of existence and stability conditions for the steady states depending on the parameter values of $a_1$ and $a_2$. "stable" refers to local asymptotic stability.}
	\label{tabelle}
\end{table}

\begin{corollary}
	{The steady state $E_2$ exists and is locally asymptotically stable if and only if }	
	\begin{align*}
	a_1&>\frac{1}{2}\\
	a_1&>a_2\\
	p_2&>\frac{1}{1-\frac{a_2}{a_1}} \, \left[\frac{1}{\mygamma} - \mybeta \, \frac{d_3}{p_1} \right]p_1.
	\end{align*}
	In this case all other non-negative steady states are unstable.  
\end{corollary}



\subsection{Hopf bifurcation}\label{sub:hopf}

In this section, we will further investigate the change of the dynamical behavior when $p_2$ passes through $\frac{1}{1-\frac{a_2}{a_1}} \, \left[\frac{1}{\mygamma} - \mybeta \, \frac{d_3}{p_1} \right]p_1$. We note that this is only possible if $\frac{1}{1-\frac{a_2}{a_1}} \, \left[\frac{1}{\mygamma} - \mybeta \, \frac{d_3}{p_1} \right]p_1>0$ holds, i.e. for $d_3 < d_3^\text{max} := \frac{p_1}{\mybeta \, \mygamma}$.

\begin{theorem}
	Let $d_3^\text{max} := \frac{p_1}{\mybeta \, \mygamma}$ and $d_3 < d_3^\text{max}$. Then the steady state $E_2$ undergoes a Hopf bifurcation with bifurcation point $p_2= p_2^* :=\frac{1}{1-\frac{a_2}{a_1}} \, \left[\frac{1}{\mygamma} - \mybeta \, \frac{d_3}{p_1} \right]p_1$, i.e.\ the Jacobian matrix $J$ at the positive steady state $E_2$  has two eigenvalues $\lambda_1, \lambda_2$ for which the following relations hold 
	\[\lambda_{1,2}(p_2)=\mu(p_2)\pm \omega(p_2), \quad \omega(p_2^*)\neq 0 ,\quad \mu(p_2^*)= 0, \quad \frac{d}{d\,p_2}\mu(p_2^*)\neq 0.\]
\end{theorem} 

\begin{proof}
	We consider the system \eqref{M1*}-\eqref{M3*}. We recall that existence of $E_2$ requires $a_1>0.5$ and  $a_2<a_1$. Let $P(x)=x^3 + b_1\,x^2 + b_2\,x + b_3$ be the characteristic polynomial of $J$. From the proof of Theorem \ref{thm:stabE2} we know that $J$ has two purely imaginary eigenvalues unequal to zero if and only if $b_1\,b_2 - b_3 =0$, i.e., $p_2 = p_2^*$.  Thus, for $\lambda_{1,2}(p_2):=\mu(p_2)\pm \omega(p_2)$ it holds  $\mu(p_2^*)= 0$, $\omega(p_2^*)\neq 0$. It remains to show that $\mu'(p_2^*)\neq 0$.\\
	
	Let us rewrite the characteristic polynomial $P(x)$ as \newline $P(x)=\left(x-\lambda_1\right)\,\left(x-\lambda_2\right)\,\left(x-\lambda_3\right)$, where $\lambda_3$ denotes the third eigenvalue. We can verify easily that the relation $b_1\,b_2-b_3=-P(\lambda_1+\lambda_2+\lambda_3)$ holds. We will obtain an expression for $\mu'(p_2^*)$ by computing the derivatives of both sides of this equation with respect to $p_2$.
	\begin{align*}
	&\left. \frac{d}{dp_2} P\left(\lambda_1\left(p_2\right) + \lambda_2\left(p_2\right) + \lambda_3\left(p_2\right)\right) \right|_{p_2 = p_2^*} \\
	= &\left. \frac{d}{dp_2} \left[ \left(\mu\left(p_2\right) - i \, \omega\left(p_2\right) + \lambda_3\left(p_2\right) \right) \,\left(\mu\left(p_2\right) + i \, \omega\left(p_2\right) + \lambda_3\left(p_2\right) \right)\cdot 2 \, \mu\left(p_2\right)\right]  \right|_{p_2 = p_2^*}\\
	= & 2 \, \mu ' \left(p_2^* \right) \, \left(\lambda_3^2\left(p_2^*\right) + \omega^2\left(p_2^*\right)\right)
	\end{align*}
	Now we calculate the derivative with respect to $p_2$ of $b_1\,b_2-b_3$:
	\begin{align*}
	&\left.\frac{d}{dp_2} \left[ -b_1\,b_2+b_3\right] \right|_{p_2 = p_2^*}\\
	= & 
	\textcolor{black}{-\left(1-\frac{1}{2\,a_1}\right)\,\left(1-\frac{a_2}{a_1}\right)^2\,d_3 \, \mygamma d_3 p_2^*}
	\end{align*}
	By equating both expressions for the derivative we obtain:
	\[\mu ' \left(p_2^* \right) =- \frac{\textcolor{black}{\left(1-\frac{1}{2\,a_1}\right)\,\left(1-\frac{a_2}{a_1}\right)^2\,d_3 \, \mygamma d_3 p_2^*}}{2 \, \left(\lambda_3^2\left(p_2^*\right) + \omega^2\left(p_2^*\right)\right)} < 0\]
	for $p_2^*>0$, since $a_1>a_2$ implies $\gamma(a_1,a_2)>0$. We note that $d_3<{d_3^{max}}$ implies $p_2^*>0$.
	since $d_3<d_3^\text{max}$. This completes the proof.
\end{proof}

\begin{remark}
	We note that the bifurcation point does not depend on $k$. 
\end{remark}

\begin{remark}
	It holds $b_1(p_2)=-(\lambda_1(p_2)+\lambda_2(p_2)+\lambda_3(p_2))$ which equals $-\lambda_3(p_2^*)$ if $p_2=p_2^*$. As we noted in the preceding proof, we have $b_1 = - \lambda_3(p_2^*)$ and therefore, we can explicitly compute the eigenvalues of the Jacobian matrix for the case $p_=p_2^*$.
	\[\lambda_3(p_2^*) = -\frac{1}{\mygamma}.\]
	Similarly, $b_3(p_2)=-\lambda_1(p_2)\lambda_2(p_2)\lambda_3(p_2)$ and $b_3 (p_2^*) =- \lambda_3\left(p_2^*\right)\omega^2\left(p_2^*\right)$, which implies
	\begin{align*}
	\omega\left(p_2^*\right)& = \sqrt{{b_3(p_2^*)}{\gamma(a_1,a_2)}} \\&= \sqrt{\left( \frac{1}{\gamma \left( a_1,a_2\right)} - \beta \left( a_1,a_2\right) \, d_3\right) \, \gamma\left(a_1, a_2\right) \, \left( 1- \frac{1}{2 \, a_1}\right) \, d_3}
	\end{align*}
	
	We note that the third eigenvalue $\lambda_3$ has a negative real part, since $\gamma(a_1, a_2)>0$. Consequently, for $p_2=p_2^*$  the orbits of the system exponentially approach a center manifold. On the center manifold, the orbits look essentially like the ones of the Hopf normal form as shown in ref. \cite{Kuz04}. 
\end{remark}
Figure \ref{figHopfOrbits} provides a numerical example for the super-critical Hopf bifurcation. 
\begin{figure}[!htb]%
	\centering
	\includegraphics[width=10cm]{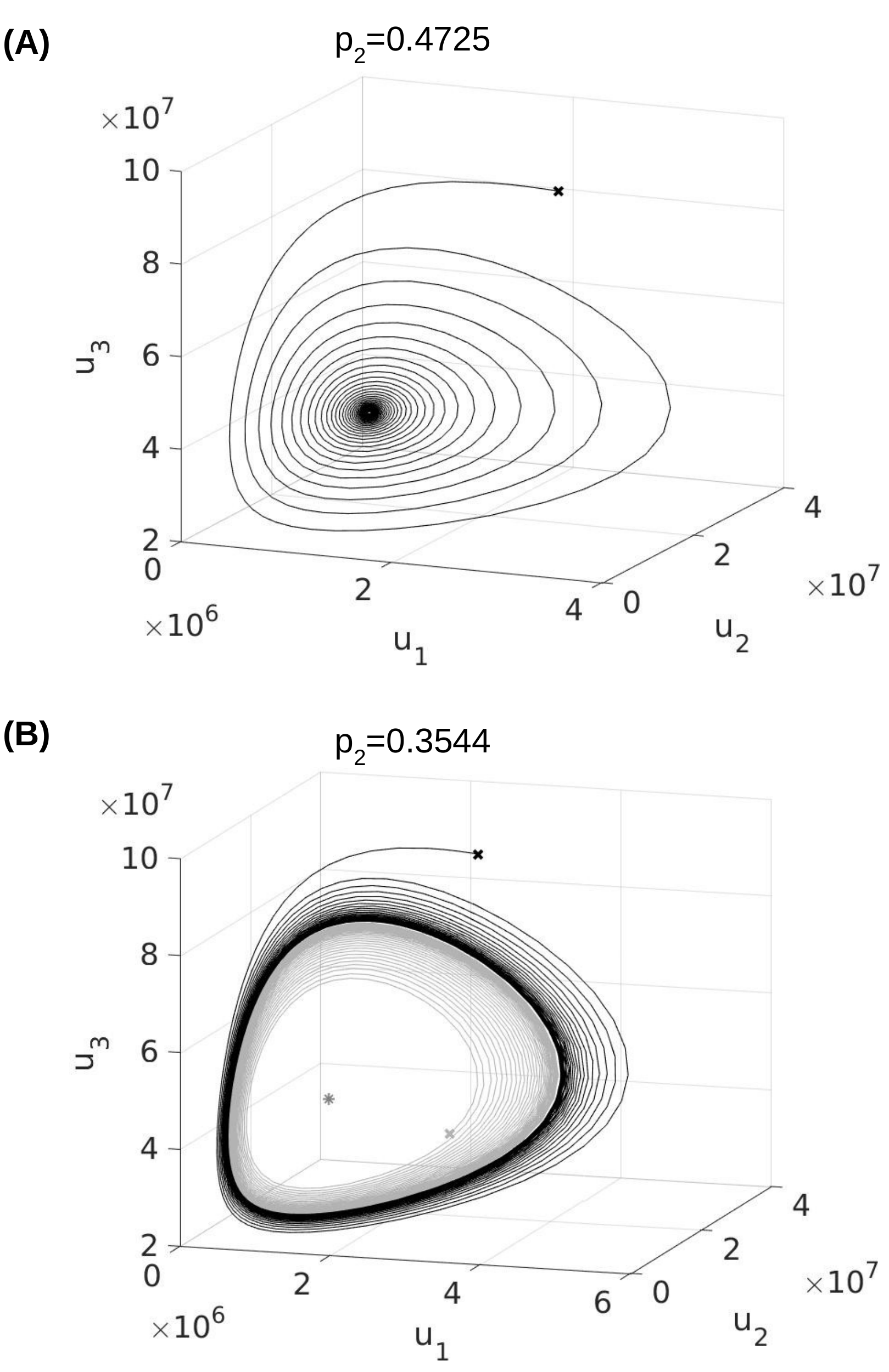}
	
	\caption[Super-critical Hopf bifurcation ]{(A) Phase portrait for $p_2>p_2^*$. The solution converges to the positive equilibrium. Initial condition: $u_1(0)=0.1766\cdot 10^7$, $u_2(0)=1.3082 \cdot 10^7$, $u_3(0)=5.9429\cdot 10^7$. (B) Phase portrait for $p_2<p_2^*$. Existence of a stable limit cycle. Initial conditions: $u_1(0)=0.2717 \cdot 10^7$, $u_2(0)=2.6836\cdot 10^7$, $u_3(0)=9.1429\cdot 10^7$ and $u_1(0)=0.1766\cdot 10^7$, $u_2(0)=1.7443 \cdot 10^7$, $u_3(0)=5.9429\cdot 10^7$. Initial conditions are marked by crosses, the positive equilibrium is marked by "*". Parameters: $a_1=0.7$, $a_2=0.5$, $p_1=1$, $d_3=0.1337$, $k_s=8.75\cdot10^{-9}.$ The Hopf bifurcation occurs at $p_2^*=0.3937$. }
	\label{figHopfOrbits}
\end{figure}

\begin{figure}
	\includegraphics[width=\textwidth]{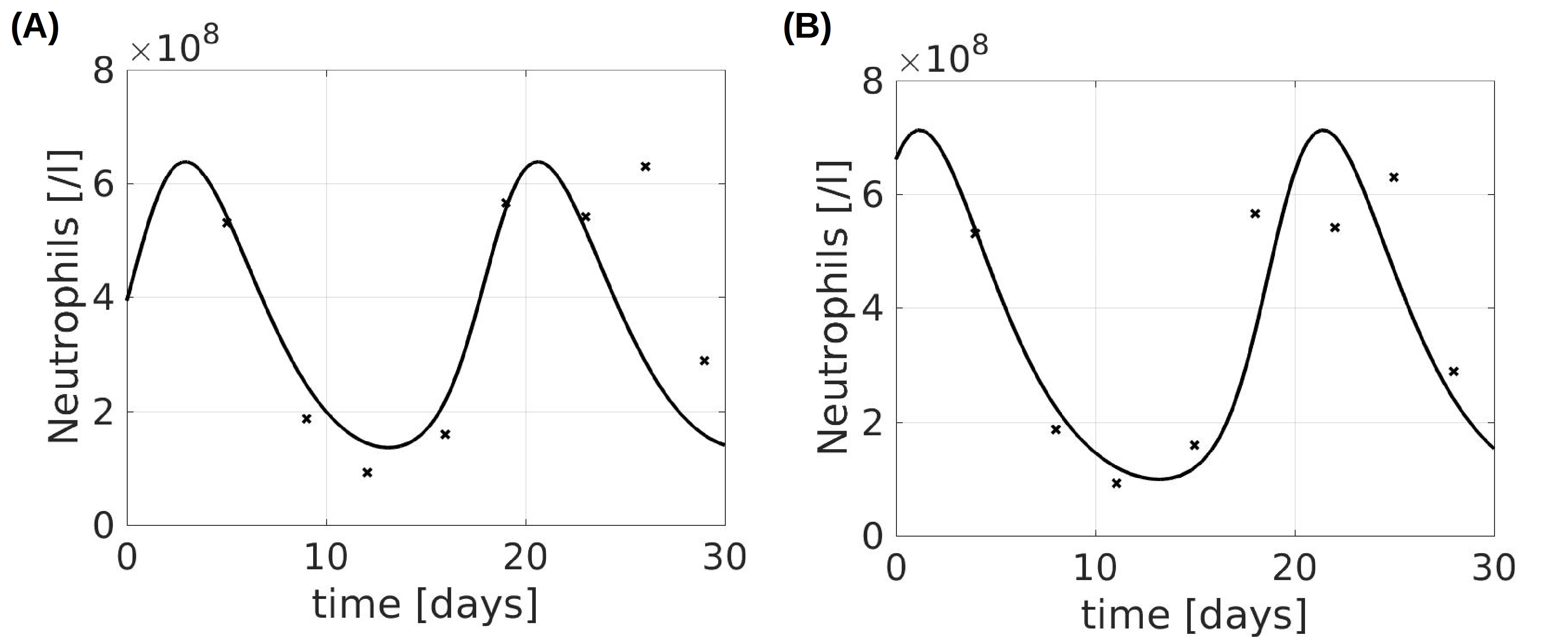}
	\caption{Overlay of patient data and example simulations: Modle simulations qualitatively reproduce neutrophil dynamics in cyclic neutropenia. Patient data taken from ref. \cite{Dale}, Fig. 2A. Parameters: (A) $a_1=0.85$, $p_1=1$, $a_2=0.841/day$, $p_2=0.4/day$, $d_1=0$, $d_2=0.5592/day$, $d_3=0.36765/day$, $k=3.5\cdot 10^{-8}$. (B) $a_1=0.85$, $p_1=0.9293$, $a_2=0.841/day$, $p_2=0.0150/day$, $d_1=0$, $d_2=0.2541/day$, $d_3=2.3/day$, $k=3.2\cdot 10^{-8}$. \label{Fig:CN}}
\end{figure}

\begin{figure}
\centering
\includegraphics[width=\linewidth]{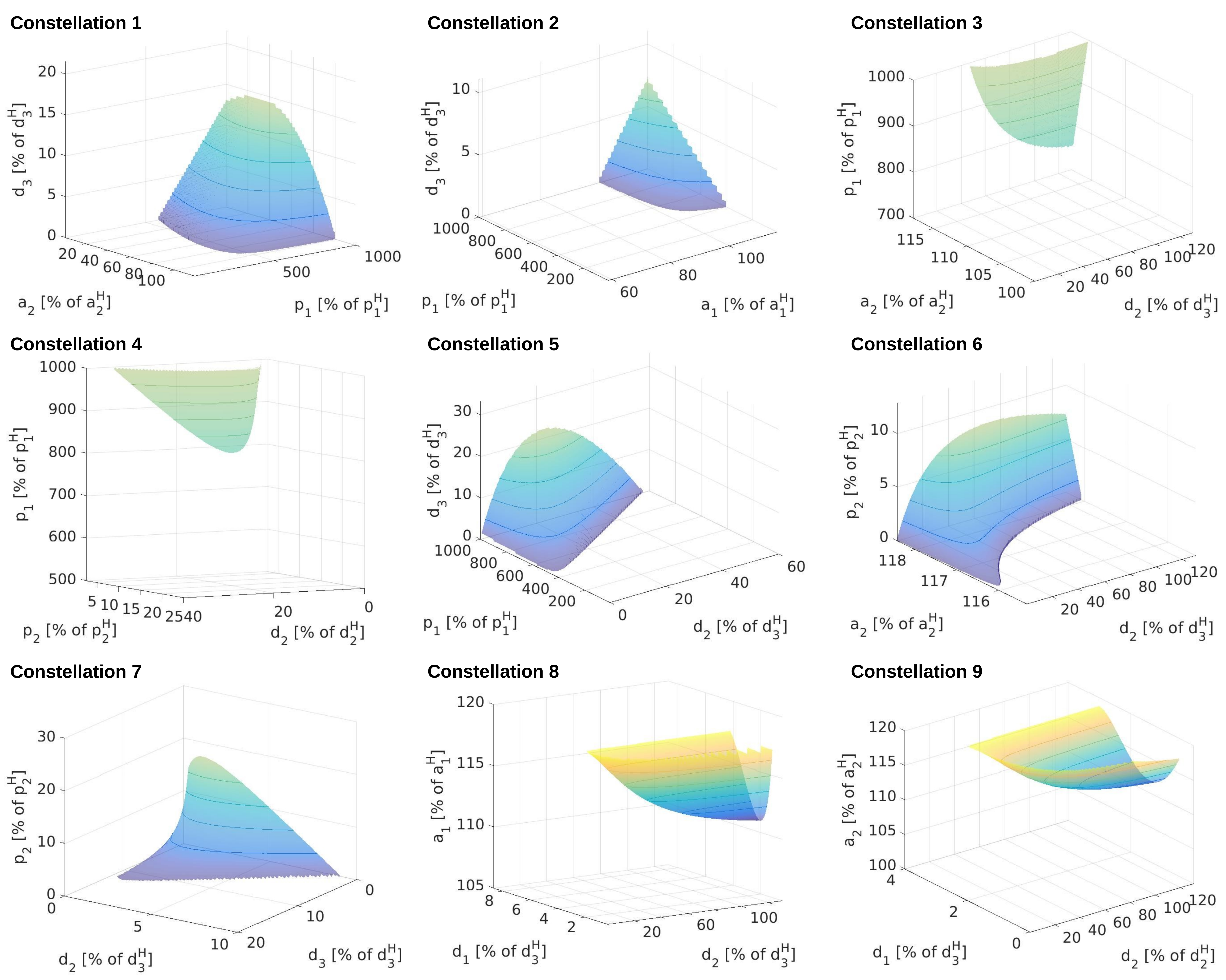}
\caption{{Parameter regions where oscillating solutions exist. Explicit examples for parameter values are provided in Appendix \ref{AppPar}. Parameter values are presented as percentages of the healthy parameters: $a_1^H$ - self-renewal of healthy stem cells,  $a_1^H$ - self-renewal of healthy progenitor cells,  $p_1^H$ - proliferation rate  of healthy stem cells, $p_2^H$ - proliferation rate  of healthy progenitor cells, $d_3^H$ - clearance rate of healthy granulocytes. Since in the healthy system immature cells do not die ($d_1^H=d_2^H=0$), $d_1$ and $d_2$ are given as percentages of $d_3^H$.} }
\label{Fig:Sets}
\end{figure}


\section{Biological implications}
In this section we apply the mathematical results to gain insights into the origin of oscillating blood cell counts. We use our model to systematically study which parameters have to deviate from their physiological values to obtain persistent oscillations. 

\subsection{Numerical studies}
Parameters of model \eqref{M1}-\eqref{M3} have been obtained using a combination of data from literature and patient data. It has been shown that the calibrated model can reproduce blood cell dynamics after bone marrow transplantation and during acute leukemia \cite{CR,BMT,AdvExpMedBiol,Interface,SciRep}. The parameters have been estimated as follows \cite{CR}:\\

$a_1=0.850$,  $p_1=0.1/day$, $a_2=0.841$, $p_2=0.4/day$, $d_1=0.0$,  $d_2=0.0$, $d_3=2.7/day$, $k=1.75\cdot 10^{-9}.$\\

In the following we refer to these parameters as {\it reference values}. To systematically check whether variation of one of these parameters can lead to a Hopf bifurcation, we search, where in the parameter space $p_2$ passes through the bifurcation point $p_2^*$. Biologically plausible intervals for the respective parameters are:

\begin{align*}
a_1&\in(0.5,\ 1)\\
p_1&\in (0/day,\ 1/day)\\
a_2&\in(0,\ 1)\\
p_2&\in (0/day,\ 1/day)\\
d_1&\in(0/day,\ 3/day)\\
d_2&\in(0/day,\ 3/day)\\
d_3&\in(0.1/day,\ 3/day)\\
\end{align*} 

These ranges are motivated as follows. The self-renewal fractions $a_1$ and $a_2$ correspond to the probabilities that a progeny cell belongs to the same maturation stage as its parent cell, therefore they assume values between zero and one. If $a_1\leq0.5$, the stem cell population declines over time \cite{MCM}, which contradicts biological observations \cite{MCM,COISB}. For this reason we assume $a_1>0.5$. A proliferation rate of $1/day$ corresponds to more than one cell division per day, which is the biological upper limit required for genome duplication \cite{AlbertsCC}. The value of $d_3$ corresponds to a {neurophils half life} between $5$ hours and 7 $days$ what is in agreement with measurements in humans \cite{Cartwright,Pillay}. Of note biological studies suggest, that in cyclic neutropenia {neutrophils half life} does not significantly differ from reference values \cite{Guerry}. As neutrophils are considered as one of the cell types with shortest half life, we assume that  $5$ hours is a reasonable lower bound also for the half life of immature cells \cite{Jandl}. For our numerical studies we subdivide the given intervals in 100 equidistant points. \\

There is biological evidence that oscillating blood cell counts may be related to death of immature cells \cite{Grenda}. For this reason we extend our numerical analysis to the case where stem and progenitor cell death rates, denoted as $d_1$ and $d_2$, can be positive. In this case we obtain the following expressions for $b_1$, $b_2$ and $b_3$.

\begin{align*}
b_1&=d_2-\left(\frac{a_2}{a_1}\frac{d_1+p_1}{p_1}-1\right)p_2+d_3\left(1+\frac{   \left(\left(1-\frac{1}{2a_1}\right)p_1-\frac{d_1}{2a_1}\right)\frac{a_2}{a_1}\frac{d_1+p_1}{p_1}  }{  \left( \frac{a_2}{a_1}\frac{d_1+p_1}{p_1}-2  \right)p_1   }\right)\\
b_2&= +d_3\left( \left(1-\frac{1}{2a_1}\right)p_1-\frac{d_1}{2a_1}\right)\frac{d_1+p_1}{p_1}\left( \frac{a_2p_2}{a_1p_1}+\frac{p_2\left(\frac{a_2}{a_1}\frac{d_1+p_1}{p_1}-1\right)-d_2}{p_1-d_1}     \right)\\
&\quad -\left(p_2\left(\frac{a_2}{a_1}\frac{d_1+p_1}{p_1}-1\right)-d_2\right)d_3\left(1+\frac{   \left(\left(1-\frac{1}{2a_1}\right)p_1-\frac{d_1}{2a_1}\right)\frac{a_2}{a_1}\frac{d_1+p_1}{p_1}  }{  \left( \frac{a_2}{a_1}\frac{d_1+p_1}{p_1}-2  \right)p_1   }\right)\\
b_3&=-\left(p_2\left(\frac{a_2}{a_1}\frac{d_1+p_1}{p_1}-1\right)-d_2\right)\left(\left(1-\frac{1}{2a_1}\right)p_1-\frac{d_1}{2a_1}\right)\frac{d_1+p_1}{p_1}d_3
\end{align*}

We numerically check at which locations of the parameter space $b_1b_2-b_3$ changes its sign from positive to negative. In the considered parameter space $b_1$ and $b_3$ are positive.

\begin{remark}
	{Global boundedness of solutions also holds for the extended model with positive $d_1$ and $d_2$. The proof works analogously to the proof of Theorem \ref{thm:ExAndBoundedness} with $B_1:=\frac{1+p_2+d_2-d_1}{2(1-a_1)}$ and $B_2:=\frac{2B_1+d_3+p_2}{2(1-a_2)}$. The expression for the unique strictly positive steady state together with necessary and sufficient criteria for its existence are provided in ref. \cite{MCM}. The same applies to the semi-trivial equilibria.} 
\end{remark}

\subsection{Results}
Numerical simulations show that if only one parameter differs from its reference value, no Hopf-bifurcation is observed. The same holds if two parameters differ from their reference values. 
The minimal requirement for a Hopf-bifurcation to occur is that three parameters deviate from their respective reference values. All scenarios where Hopf bifurcation can occur are summarized in Table \ref{TabHopf}, more details are provided in Appendix \ref{AppPar} and {Figure \ref{Fig:Sets}}. {For some parameters significant deviations from the reference values are required to observe Hopf bifurcation. One example for this is the neutrophil apoptosis rate, where a reduction to less than 25\% of its original value is required. Taking into account the large variations of the respective cell parameters observed in healthy individuals under immune stimulation (5-10 fold), parameter changes of this order of magnitude can be considered biologically realistic \cite{Lord}.}\\

Our findings are in line with the following results from literature:
\begin{enumerate}[1.]
	\item The fact that Hopf bifurcations only occur if multiple parameters deviate from their reference values may explain why oscillating blood cell counts are rarely observed. 
	\item The death rates of immature cells are increased in most parameter configurations. This is in line with experimental findings  \cite{Grenda}. However, increase of immature cell death is not necessary to obtain a Hopf bifurcation (see Constellations 1 and 2 in Table \ref{TabHopf}). 
	\item The model can reproduce neutrophil dynamics in cyclical neutropenia, see Fig. \ref{Fig:CN}, if the feedback parameter $k$ deviates from its reference value.
	\item For certain parameter values the model shows oscillations with a period of several months. Such oscillations have been observed in case of chronic myeloid leukemia \cite{Hirayama,Rodriguez,Gatti}.
	\item Many parameter configurations show increased death rates of stem and/or progenitor cells. This finding can explain why in some patients oscillations are induced by chemotherapeutic agents and other drugs inducing cell death \cite{Baird,Kennedy,Morley2}.
\end{enumerate}

The numerical studies provide the following insights into diseases with oscillating blood cell counts.
\begin{enumerate}[1.]
	
	\item In some parameter configurations $a_1$ or $p_1$ or both are increased. Several studies show that increased stem cell self-renewal and proliferation are linked to malignancy \cite{CR,Interface,Kelly,Kikushige,Wang2}. This may explain why oscillating blood cell counts can be interpreted as a premalignant state \cite{Lensink,Dale3}. However, increased stem cell self-renewal and proliferation are not necessary for a Hopf-bifurcation to occur. This may explain why several studies do not see a relation between oscillating blood cell counts and malignancies \cite{Dale}.
	\item According to the considered model Hopf bifurcations can occur for decreased mature cell clearance $d_3$ but not for increased $d_3$.  This is surprising, since for a long time increased mature cell death has been suspected as the origin of oscillating blood cell counts. Experiments, however, have shown that mature cell clearance is not increased in cyclic neutropenia \cite{Guerry}. In some patients with cyclic neutropenia a slight decrease of neutrophil clearance has been reported  \cite{Guerry}. 
	\item Periodic auto-inflammatory syndromes, such as the PFAPA syndrome are associated with reduced neutrophil apoptosis \cite{Kraszewska-Głomba, Sundqvist}. These syndromes are characterized by periodic inflammations with increased white blood cell counts \cite{Brown}. Our results support the idea that decreased neutrophil apoptosis may contribute to the periodicity of the symptoms.
	\item All detected parameter configurations involve changes in progenitor or mature cell parameters compared to healthy controls. It is controversial whether oscillating blood cell counts require functional deficits on the level of stem cells or whether alterations in the progenitor cell compartment are sufficient. Our model suggests that both scenarios can exist. In some cases stem cell properties ($a_1$, $p_1$, $d_1$) differ from their physiologic reference values (Constellations 1-5 and 8, 9), whereas in other cases they do not (Constellations 6 and 7).
	\item Whenever a Hopf bifurcation occurs, the counts of stem, progenitor and mature cells oscillate. This is in line with bone marrow examinations showing oscillations of immature cells \cite{Guerry}.  Furthermore it fits to the clinical observation that in many patients not only white but also other blood cell counts oscillate \cite{Langlois,Guerry}. Of note the observation of oscillating stem cell counts does not imply that stem cell parameters have to deviate from their reference values.  Alterations of progenitor cell parameters can be sufficient (Constellations 6 and 7). 
	
	\item For many parameter configurations blood cell counts oscillate within physiological bounds. This may suggest that oscillating blood cell counts not necessarily lead to clinical symptoms. The occurrence of oscillating blood cell counts in healthy patients is so far controversial \cite{Morley,Dale5,Dale}.     
	\item If the value of $k$ remains unchanged, all considered parameter configurations lead to neutrophil counts that are too high to be compatible with those observed in cyclic neutropenia \cite{Dale}. However, changes in the parameter $k$ lead to the clinical observed low neutrophil counts. Examples are depicted in Figure \ref{Fig:CN}. This observation is in line with the  experimental finding that response of immature cells to feedback signals is altered in patients with cyclic neutropenia \cite{Hammond}. To observe low neutrophil counts $k$ has to be higher than its reference value. This means that the effect of cytokines is reduced in patients with cyclic neutropenia, as it has been observed experimentally  \cite{Hammond}. 
	\item The bifurcation point is independent of the value of the feedback parameter $k$. This means that changes in the feedback signal that affect all cells simultaneously  cannot produce oscillations. Consequently, substitution of feedback signals (such as G-CSF) cannot prevent oscillations. This is in line with clinical observations  \cite{Hammond3}. This finding supports the idea that alterations in the feedback mechanism may be responsible for the low neutrophil counts in cyclic neutropenia but not for the oscillations.

\end{enumerate}  
\begin{table}
	\begin{tabular}{c |c |c |c |c |c |c |c}
		Constellation&$a_1$ &$p_1$&$d_1$&$a_2$&$p_2$&$d_2$&$d_3$\\
		\hline	
		1& &$\uparrow$&$$&$\downarrow$& & &$\downarrow$\\
		2&$\uparrow$ &$\uparrow$& & & & &$\downarrow$\\
		3&  &$\uparrow$& &$\uparrow$& &$\uparrow$ &\\
		4& &$\uparrow$& & &$\downarrow$&$\uparrow$&\\
		5& &$\uparrow$& & & &$\uparrow$&$\downarrow$\\
		6&  & & &$\uparrow$&$\downarrow$&$\uparrow$&\\
		7& & & & &$\downarrow$&$\uparrow$&$\downarrow$\\
		8&$\uparrow$ & &$\uparrow$& & &$\uparrow$&\\
		9&$$ &&$\uparrow$&$\uparrow$&&$\uparrow$&\\
	\end{tabular}
	\caption{Parameter configurations leading to a Hopf bifurcation\label{TabHopf}}
\end{table}

\section{Discussion}\label{sec:instab}
In this work we prove the occurrence of Hopf bifurcation in a mathematical model of white blood cell formation. The  model describes time evolution of a cell population structured by three maturation stages (stem, progenitor and mature cells) that are regulated by a nonlinear feedback mechanism.  We show that for appropriate parameter choices a super-critical Hopf bifurcation occurs and a stable limit cycle emerges. This constitutes a major difference to the two-compartment version of the model that distinguishes only between mature and immature cells and has a globally stable steady state \cite{Nakata}. This finding demonstrates that the number of maturation stages can impact dynamical features of the system. \\

The considered model has been applied to clinical data and shows a good agreement with reality \cite{SCDev,BMT,CR,Interface,AdvExpMedBiol,SciRep}. It predicts that a Hopf bifurcation can occur for biologically plausible parameters and thus provides possible explanations for disease mechanisms that lead to oscillating blood cell counts. Our  systematic numerical study suggests that sustained oscillations only occur if multiple parameters deviate from their physiological reference values. This finding is in line with the observation that oscillating blood cell counts are rarely observed. \\

Biological data and theoretical results have linked the occurrence of oscillating blood cell counts to increased death rates of immature cells. This finding is supported by our model, however, our results suggest that increased immature cell death is not necessary to obtain a stable limit cycle. Alterations of mature cell death rates together with stem or progenitor cell self-renewal and proliferation are also sufficient. The impact of perturbed self-renewal on oscillating blood cell counts has been discussed in the context of a linear model  \cite{Dingli}.\\

One of the most common diseases exhibiting periodic oscillations of white blood cells it cyclic neutropenia. Our model suggests that the low cell numbers detected in these patients can only be reproduced if the response of immature cells to feedback signals is reduced. This result is in line with in vitro experimental findings. Interestingly, the occurrence of Hopf-bifurcation is independent of the parameters that describe the feedback signal. This suggests that alterations in the feedback that simultaneously affect all cell types cannot lead to oscillating cell counts. However, in presence of  parameter configurations that lead to oscillations, an additional alteration of the feedback signal can impact their amplitude. This result leads to the hypothesis that the experimentally detected reduced response of immature cells to cytokines in patients with cyclic neutropenia is the pathogenic mechanism leading to low cell counts; however it does not contribute to the occurrence of oscillations. This hypothesis is further supported by clinical data showing that administration of growth factors such as G-CSF increases white blood cell counts but does not lead to cessation of the periodic oscillations \cite{Migliaccio,Hammond3}.   \\

The finding that multiple different parameter configurations can result in oscillating blood cell counts may explain parts of the heterogeneity among cyclic neutropenia patients. It furthermore suggests that different detected mutations \cite{Makaryan,Germeshausen,Alangari,Boo,Whited,Cipe} may lead to different pathogenic mechanisms that result in similar symptoms.\\

In summary, we have proven the occurrence of a Hopf bifurcation in a non-linear three-compartment model of white blood cell formation. The Hopf bifurcation is a unique feature of the three-compartment setting and does not occur in the  2-compartment version of the model. We identify biologically plausible parameter sets that lead to a stable limit cycle and relate them to clinical and experimental findings. This quantitative approach can help to understand the pathogenic mechanisms and the clinical heterogeneity of different diseases that lead to oscillating blood cell counts.






\appendix

\section{Supplementary calculations to the proof of Theorem \ref{thm:stabE2} }\label{app:proofsstab}

In the following we provide the calculations leading to equation \eqref{Hurwitz}.  

\begin{align*}
& \, b_1\,b_2 - b_3\\
= &\left[\left(1-\frac{a_2}{a_1}\right)\,p_2 +  		\left(1-\frac{a_2}{a_1}\,\left(1-\frac{1}{2\,a_1}\right)\,\frac{1}{2-\frac{a_2}{a_1}}\right)\,d_3\right]\\
&\cdot \left[\left(1-\frac{a_2}{a_1}\right)\,\left(1-\frac{a_2}{a_1}\,\left(1-\frac{1}{2\,a_1}\right)\,\frac{1}{2-\frac{a_2}{a_1}}\right) - \left(1-\frac{1}{2\,a_1}\right)\,\left(1-2\,\frac{a_2}{a_1}\right)\right]\,d_3\,p_2\\
&- \left(1-\frac{1}{2\,a_1}\right)\,\left(1-\frac{a_2}{a_1}\right)\,d_3\,p_2\\ \displaybreak[0]
=& \left[\left(1-\frac{a_2}{a_1}\right)\,p_2 +  		\left(1-\frac{a_2}{a_1}\,\left(1-\frac{1}{2\,a_1}\right)\,\frac{1}{2-\frac{a_2}{a_1}}\right)\,d_3\right]\\
&\cdot \left(1-\frac{1}{2\,a_1}\right)\,\left(1-\frac{a_2}{a_1}\right)\, \left[\frac{1}{1-\frac{1}{2\,a_1}} - \frac{a_2}{a_1}\,\frac{1}{2-\frac{a_2}{a_1}} - \frac{1-2\,\frac{a_2}{a_1}}{1-\frac{a_2}{a_1}} \right]\,d_3\,p_2\\
&- \left(1-\frac{1}{2\,a_1}\right)\,\left(1-\frac{a_2}{a_1}\right)\,d_3\,p_2\\ \displaybreak[0]
=& \left(1-\frac{1}{2\,a_1}\right)\,\left(1-\frac{a_2}{a_1}\right)\,\left(\left[\left(1-\frac{a_2}{a_1}\right)\,p_2 +  		\left(1-\frac{a_2}{a_1}\,\left(1-\frac{1}{2\,a_1}\right)\,\frac{1}{2-\frac{a_2}{a_1}}\right)\,d_3\right] \right.\\
&\left.\cdot \left[\frac{1}{2\,a_1}\frac{1}{1-\frac{1}{2\,a_1}} + \frac{a_2}{a_1}\,\frac{1}{\left(2-\frac{a_2}{a_1}\right)\,\left(1-\frac{a_2}{a_1}\right)} \right] -1\right)\,d_3\,p_2\\
=& \left(1-\frac{1}{2\,a_1}\right)\,\left(1-\frac{a_2}{a_1}\right)\,\left(\left[\left(1-\frac{a_2}{a_1}\right)\,p_2 +  		\mybeta \,d_3\right] \cdot \mygamma -1\right)\,d_3\,p_2.
\end{align*}
\newpage

\section{Parameter configurations leading to Hopf bifurcation}\label{AppPar}

\begin{tabular}{p{8cm}p{8cm}}
	{\bf Constellation 1:}\newline
	$p_1$ increased\newline
	$a_2$ decreased \newline
	$d_3$ decreased ($<0.5/day$)\newline
	Example: \begin{itemize}
		\item []$p_1=0.7171/day$
		\item [] $a_2=0.32$
		\item [] $d_3=0.132/day$
	\end{itemize}
	
	&	
	{\bf Constellation 2:}\newline
	$p_1$ increased\newline
	$a_1$ increased\newline
	$d_3$ decreased (close to $0.1/day$)\newline
	Example: 
	\begin{itemize}
		\item [] $p_1=0.9697/day$
		\item [] $a_1=0.99$
		\item [] $d_3=0.132/day$
	\end{itemize}
	\\
	{\bf Constellation 3:}\newline
	$p_1$ increased\newline
	$a_2$ increased (close to 1)\newline
	$d_2$ increased\newline
	Example:
	\begin{itemize}
		\item [] $p_1=0.7778/day$
		\item [] $a_2=0.99$
		\item []$d_2=2.6644/day$
	\end{itemize} 
	
	&
	{\bf Constellation 4:}\newline
	$p_1$ increased\newline
	$p_2$ decreased\newline
	$d_2$ increased\newline
	Example:
	\begin{itemize}
		\item [] $p_1=0.8687/day$
		\item [] $p_2=0.0201/day$
		\item [] $d_2=0.2541/day$
	\end{itemize}  
	\\

	{\bf Constellation 5:}\newline
	$p_1$ increased\newline
	$d_2$ increased \newline
	$d_3$ decreased \newline
	Example:
	\begin{itemize}
		\item [] $p_1=0.707/day$
		\item [] $d_2=0.2541/day$
		\item [] $d_3=0.132/day$
	\end{itemize}  
	
	&
	{\bf Constellation 6:}\newline
	$p_2$ decreased (close to 0.01)\newline
	$a_2$ increased (close to 1) \newline
	$d_2$ increased \newline
	Example:
	\begin{itemize}
		\item [] $p_2=0.01/day$
		\item []$a_2=0.99$
		\item []$ d_2=0.5287/day$
	\end{itemize} 

\end{tabular}
\pagebreak

\begin{tabular}{p{8cm}p{8cm}}
	
	{\bf Constellation 7:}\newline
	$p_2$ decreased\newline
	$d_2$ increased \newline
	$d_3$ decreased \newline
	Example: 
	\begin{itemize}
		\item [] $p_2=0.01/day$
		\item [] $d_2=0.0405/day$
		\item [] $d_3=0.132/day$
	\end{itemize} 
	
	&
	{\bf Constellation 8:}\newline
	$a_1$ increased\newline
	$d_1$ slightly increased ($<0.1/day$)\newline
	$d_2$ increased \newline
	Example:
	\begin{itemize}
		\item [] $a_1=0.95$
		\item [] $d_1=0.0405/day$
		\item [] $d_2=2.7559/day $
	\end{itemize} 
	
	\\
	
	{\bf Constellation 9:}\newline
	$a_2$ increased\newline
	$d_1$ slightly increased ($<0.05/day$)\newline
	$d_2$ increased \newline
	Example: 
	\begin{itemize}
		\item [] $a_2=0.95$
		\item []$d_1=0.0405/day$
		\item []$d_2=2.5423/day$
	\end{itemize} 
	
\end{tabular}
{Biologically plausible parameter regions where the Hopf-bifurcation exists are visualized in Figure \ref{Fig:Sets}}. {The reported values for $a_1$ and $a_2$ correspond to the self-renewal fraction in presence of maximal stimulation. The self-renewal fraction at time $t$ is given by $a_1s(t)$ and $a_2s(t)$ with $s(t)<1$.} 

\section*{Acknowledgments}
This work was supported by research funding from the German Research Foundation DFG (Collaborative Research Center SFB 873, Maintenance and Differentiation of Stem Cells in Development and Disease, subproject B08).   


\include{Literaturverzeichnis}

%
%

\bibliographystyle{abbrv}   
\bibliography{Bibliography.bib}   

\end{document}